\documentclass{amsart}
\usepackage{amsmath,amsfonts,amssymb,amsthm}
\usepackage[all]{xy}
\usepackage{tikz}
\usepackage{comment}

\newcommand{\uid}{
\begin{tikzpicture}[scale=0.55]
\draw[line width = 0.06cm] (0,-0.01) -- (0,2.02);
\draw[line width = 0.06cm] (1,-0.01) -- (1,2.02);
\end{tikzpicture}}

\newcommand{\puuv[1]}{
\begin{tikzpicture}[scale=0.55]
\draw[line width = 0.06cm] (0,0.3) -- (0,0.8);
\draw[line width = 0.06cm] (1,0.3) -- (1,0.8);
\draw(0.52,0.55) node {$#1$};
\draw[line width = 0.06cm] (0,0.55) -- (0.25,0.55);
\draw[line width = 0.06cm] (0.75,0.55) -- (1,0.55);
\draw[->, line width = 0.06cm] (0.5,0.9) -- (0.5,1.3);
\draw[->, line width = 0.06cm, white] (0.5,0.2) -- (0.5,-0.2);
\end{tikzpicture}
}

\newcommand{\Dtst}{
 \begin{tikzpicture}[scale=0.55]
\fill[black!35!white] (0,3.01) .. controls (0,2) and (0.5,2).. (1.5,1.5) ..  controls (1,2) and (1,2.5) .. (1,3.01);
\fill[black!35!white] (3,3.01) .. controls (3,2) and (2.5,2) .. (1.5,1.5) .. controls (2,2) and (2,2.5) .. (2,3.01);
\draw[line width = 0.07cm, white] (1.1,1.69) -- (1.9, 1.45);
\draw[line width = 0.06cm] (0,3) .. controls (0,2) and (0.5,2) .. (1.5,1.5) .. controls (2.5,1) and (3,1) .. (3,0);
\draw[line width = 0.08cm, white] (1.1,1.69) -- (1.9, 1.31);
\draw[line width = 0.06cm] (1,3) .. controls (1,2.5) and (1,2) .. (1.5,1.5) .. controls (2,1) and (2,0.5) .. (2,0);
\draw[line width = 0.08cm, white] (1.35,1.66) -- (1.65, 1.33);
\draw[line width = 0.06cm] (2,3) .. controls (2,2.5) and (2,2) .. (1.5,1.5) .. controls (1,1) and (1,0.5) .. (1,0);
\draw[line width = 0.07cm, white] (1.625,1.654) -- (1.383, 1.355);
\draw[line width = 0.06cm] (3,3) .. controls (3,2) and (2.5,2) .. (1.5,1.5) .. controls (0.5,1) and (0,1) .. (0,0);
  \end{tikzpicture}
}

\newcommand{\ptwist}{
 \begin{tikzpicture}[scale=0.55]
\fill[black!35!white] (0,-0.01) .. controls (0,0.5) and (0,0.5) .. (0.5,1) .. controls (1,0.5) and (1,0.5) .. (1,-0.01);
\draw[line width = 0.06cm] (0,-0.01) .. controls (0,0.5) and (0,0.5) .. (0.5,1) .. controls (1,1.5) and (1,1.5) .. (1,2);
\draw[line width = 0.07cm, white] (0.4,0.9) .. controls (0.5,1) and (0.5,1) .. (0.6,1.1);
\draw[line width=0.06cm] (1,-0.01) .. controls (1,0.5) and (1,0.5) .. (0.5,1) .. controls (0,1.5) and (0,1.5) .. (0,2);
  \end{tikzpicture}}

\newcommand{\pplaincrossing}{
\begin{tikzpicture}[scale=0.55]
\draw[line width = 0.06cm] (0,3) .. controls (0,2) and (2,1) .. (2,0);
\draw[line width = 0.06cm] (1,3) .. controls (1,2) and (3,1) .. (3,0);
\draw[line width = 0.8cm, white] (0.5,0) .. controls (0.5,1) and (2.5,2) .. (2.5,3);
\draw[line width = 0.06cm] (0,0) .. controls (0,1) and (2,2) .. (2,3);
\draw[line width = 0.06cm] (1,0) .. controls (1,1) and (3,2) .. (3,3);
\end{tikzpicture}}

\newtheorem{Theorem}{Theorem}[section]
\newtheorem{Proposition}[Theorem]{Proposition} 
\newtheorem{Lemma}[Theorem]{Lemma}

\newtheorem{Corollary}[Theorem]{Corollary}

\newtheorem{Definition}[Theorem]{Definition}
\theoremstyle{definition}

\newtheorem{Motivation}{Motivation}
\newtheorem{Comment}[Theorem]{Comment}

\newcommand{\wt}{\operatorname{wt}}

\newcommand{\bc}{\Bbb{C}}
\newcommand{\bq}{\Bbb{Q}}

\newcommand{\g}{\mathfrak{g}}

\newcommand{\Flip}{\mbox{Flip}}

\newcommand{\Id}{\mbox{Id}}

\newcommand{\C}{\mathcal{C}}

\newcommand{\barr}{\text{bar}}

\theoremstyle{definition}

\begin{document}

\title[A formula for the $R$-matrix]{A half-twist type formula for the $R$-matrix of a symmetrizable Kac-Moody algebra}

\author{Peter Tingley}
\email{P.Tingley@ms.unimelb.edu.au}
\address{Department of Mathematics, University of Melbourne, Parkville, VIC, 3010,  Australia}

\begin{abstract}
Kirillov-Reshetikhin and  Levendorskii-Soibelman developed a formula for the universal $R$-matrix of $U_q(\g)$ of the form $R=(X^{-1} \otimes X^{-1}) \Delta(X)$. The action of $X$ on a representation $V$ permutes weight spaces according to the longest element in the Weyl group, so is only defined when $\g$ is of finite type. We give a similar formula which is valid for any symmetrizable Kac-Moody algebra. This is done by replacing the action of $X$ on $V$ with an endomorphism that preserves weight spaces, but which is bar-linear instead of linear. \end{abstract}

\maketitle 


\section{Introduction}

Let $\g$ be a finite type complex simple Lie algebra, and let $U_q(\g)$ be the corresponding quantized universal enveloping algebra. In \cite{KR:1990} and \cite{LS}, Kirillov-Reshetikhin and Levendorskii-Soibelman developed a formula for the universal R-matrix 
\begin{equation} \label{KReq} R= (X^{-1} \otimes X^{-1}) \Delta(X), \end{equation}
where $X$ belongs to a completion of $U_q(\g)$. The element $X$ is constructed using the braid group element $T_{w_0}$ corresponding to the longest word of the Weyl group, so only makes sense when $\g$ is of finite type. 

The element $X$ defines a vector space endomorphism $X_V$ on each representation $V$, and in fact $X$ is defined by this system $\{ X_V \}$ of endomorphisms. With this point of view, Equation (\ref{KReq}) is equivalent to the claim that, for any finite dimensional representations $V$ and $W$ and $u \in V \otimes W$,
\begin{equation} \label{KReq2}
R (u) =(X_V^{-1} \otimes X_W^{-1}) X_{V \otimes W} (u).
\end{equation}

In the present work we replace $X_V$ with an endomorphism $\Theta_V$ which preserves weight spaces. We show that, for any symmetrizable Kac-Moody algebra $\g$, and any integrable highest weight representations $V$ and $W$ of $U_q(\g)$, the action of the universal $R$-matrix on $u \in V \otimes W$ is given by
\begin{equation} \label{teq1}
R (u)= (\Theta_V^{-1} \otimes \Theta_W^{-1}) \Theta_{V \otimes W} (u). 
\end{equation}
There is a technical difficulty because $\Theta_V$ is not linear over the base field $\bq(q)$, but instead is compatible with the automorphism of $\bq(q)$ which inverts $q$. For this reason $\Theta_V$ depends on a choice of a ``bar involution" on $V$. To make Equation (\ref{teq1}) precise we define a bar involution on $V \otimes W$ in terms of chosen involutions of $V$ and $W$, and then show that the composition  $(\Theta_V^{-1} \otimes \Theta_W^{-1}) \Theta_{V \otimes W}$ does not depend on any choices.

The system of endomorphisms $\Theta$ was previously studied in \cite{RTheta}, where it was used to construct the universal $R$-matrix when $\g$ is of finite type. Essentially we have extended this previous work to include all symmetrizable Kac-Moody algebras. However, the action of $\Theta$ on a tensor product is defined differently here than in  \cite{RTheta}, so the constructions of $R$ are a-priori not identical, and we have not in fact proven that the construction in \cite{RTheta} gives the universal $R$-matrix in all cases. 

This note is organized as follows. In Section \ref{background} we establish notation and review some background material. In Section \ref{maketheta} we construct the system of endomorphisms $\Theta$. In Section \ref{main} prove our main Theorem (Theorem \ref{main_th}), which simply says that our construction gives the universal $R$-matrix in all cases. In Section \ref{Questions} we discuss two questions which motivated this work. 

\subsection{Acknowledgements}
We thank Joel Kamnitzer, Nicolai Reshetikhin and Noah Snyder for many helpful discussions. 
This work was partially supported by the NSF RTG grant DMS-035432 and the Australia Research Council grant DP0879951.

\section{Background} \label{background}

\subsection{Conventions} \label{notation}
We first fix some notation. For the most part we follow conventions from \cite{CP}.

$\bullet$ $\g$ is a complex simple Lie algebra with Cartan algebra $ \mathfrak{h} $ and Cartan matrix $A = (a_{ij})_{i,j \in I}$. 

$\bullet$ $ \langle \cdot , \cdot \rangle $ denotes the paring between $ \mathfrak{h} $ and $ \mathfrak{h}^\star $ and $ ( \cdot , \cdot) $ denotes the usual symmetric bilinear form on either $ \mathfrak{h}$ or $ \mathfrak{h}^\star $.  Fix the usual bases $ \alpha_i $ for $ \mathfrak{h}^\star $ and $ H_i $ for $\mathfrak{h}$, and recall that $ \langle H_i, \alpha_j \rangle = a_{ij} $.  

$\bullet$ $ d_i = (\alpha_i, \alpha_i)/2 $, so that $ (H_i, H_j) = d_j^{-1} a_{ij} $. 

$\bullet$ $\rho$ is the weight satisfying $(\alpha_i, \rho)=d_i$ for all $i$.

$\bullet$ $U_q(\g)$ is the quantized universal enveloping algebra associated to $\g$, generated over $\bq (q)$ by $E_i$ and $F_i$ for all  $i \in I$, and $K_w$ for $w$ in the co-weight lattice of $\g$. As usual, let $K_i= K_{H_i}.$ We use conventions as in \cite{CP}. For convenience, we recall the exact formula for the coproduct:
\begin{equation} \label{coproduct}
\begin{cases}
\Delta{E_i} & = E_i \otimes K_i + 1 \otimes E_i \\
\Delta{F_i} &= F_i \otimes 1 + K_i^{-1} \otimes F_i \\
\Delta{K_i} &= K_i \otimes K_i
\end{cases}
\end{equation}

$\bullet$ We in fact need to adjoint a fixed $k^{th}$ root of $q$ to $\bq(q)$, where $k$ is twice the size of the weight lattice mod the root lattice. We denote this by $q^{1/k}$.

$\bullet$ $V_\lambda$ is the irreducible representation of $U_q(\g)$ with highest weight $\lambda$. 

$\bullet$ $v_\lambda$ is a highest weight vector of $V_\lambda$.

$\bullet$ A vector $v$ in a representation $V$ is called {\it singular} if $E_i(v)=0$ for all $i \in I$.

$\bullet$ $V(\mu)$ denotes the $\mu$ weight space of $V$.

$\bullet$ Throughout, a representation of $U_q(\g)$ means a type 1 integrable highest weight representation.

\subsection{The R-matrix} \label{Rstuff}

We briefly recall the definition of a universal $R$-matrix, and the related notion of a braiding.

\begin{Definition} 
A braided monoidal category is a monoidal category $\C$, along with a natural system of isomorphisms $\sigma^{br}_{V,W}: V \otimes W \rightarrow W \otimes V$ for each pair $V,W \in \C$, such that, for any $U,V, W \in \C$, the following two equalities hold:
\begin{equation} \sigma^{br}_{U,W} \otimes \Id \circ \Id \otimes \sigma^{br}_{V,W} = \sigma^{br}_{U \otimes V, W} \end{equation}
\begin{equation} \Id \otimes \sigma^{br}_{U,W} \circ \sigma^{br}_{U,V} \otimes  \Id = \sigma^{br}_{U,V\otimes W}. \end{equation}
 The system $\sigma^{br} : = \{ \sigma^{br}_{V,W} \}$ is called a braiding on $\C$.
\end{Definition}

Let $\widetilde{U_q(\g) \otimes U_q(\g)}$ be the completion of $U_q(\g) \otimes U_q(\g)$ in the weak topology defined by all matrix elements of $V_\lambda \otimes V_\mu$, for all ordered pairs of deminant integral weights $(\lambda, \mu)$.

\begin{Definition} \label{rR}
A universal $R$-matrix is an element $R$ of $\widetilde{U_q(\g) \otimes U_q(\g)}$ such that $\sigma^{br}_{V,W} := \mbox{Flip} \circ R$ is a braiding on the category of $ U_q(\g) $ representations. Equivalently, an element $R$ is a universal $R$-matrix if it satisfies the following three conditions
\begin{enumerate}
\item \label{intertwine} For all $u \in U_q(\g)$, $R \Delta(u)= \Delta^{op}(u) R.$

\item \label{qtriangle1} $(\Delta \otimes 1)R = R_{13} R_{23},$ where $R_{ij}$ mean $R$ placed in the $i$ and $j^{th}$ tensor factors. 

\item \label{qtriangle2} $(1 \otimes \Delta)R = R_{13} R_{12}.$ 
\end{enumerate}
\end{Definition}

The following theorem is central to the theory of quantized universal enveloping algebra. See \cite{CP} for a discussion when $\g$ is of finite type, and \cite{Lusztig:1993} for the general case. Unfortunately the conventions in \cite{Lusztig:1993} are quite different from those used here. An explicit proof that our statement follows from \cite[Chapter 4]{Lusztig:1993} can be found at  http://www.ms.unimelb.edu.au/$\sim$ptingley/lecturenotes/RandquasiR.pdf.

\begin{Proposition} \label{goodR}
Let $\g$ be a symmetrizable Kac-Moody algebra. Then $U_q(\g)$ has a unique universal $R$-matrix of the form
\begin{equation}
R=  A {\Big (} 1 \otimes 1+ \sum_{\begin{array}{c} \txt{positive integral} \\ \text{weights } \beta \text{ (with} \\ \text{multiplicity)} \end{array} } X_\beta \otimes Y_\beta {\Big )},
\end{equation}
where $X_\beta$ has weight $\beta$, $Y_\beta$ has weight $-\beta$, and for all $v \in V$ and $w \in W$, $A(v \otimes w) = q^{(\wt(v), \wt(w))}$.
\end{Proposition}

\subsection{Constructing isomorphisms using systems of endomorphisms} \label{autint}
 In this section we review a method for constructing natural systems of isomorphisms $\sigma_{V,W}: V \otimes W \rightarrow W \otimes V$ for representations $V $ and $W$ of $U_q(\g)$. This idea was used by Henriques and Kamnitzer in \cite{cactus}, and was further developed in \cite{Rcommutor}. The data needed is:

\begin{enumerate}

\item An algebra automorphism $C_\xi$ of $U_q(\g)$ which is also a coalgebra anti-automorphism.

\item \label{diagg} A natural system of invertible (vector space) endomorphisms $\xi_V$ of each representation $V$ of $U_q(\g)$ such that the following diagram commutes for all $V$:
\begin{equation} \label{compTheta}
\xymatrix{
V  \ar@(dl,dr) \ar@/ /[rrr]^{\xi_V} &&& V \ar@(dl,dr) \\
U_q(\g)  \ar@/ /[rrr]^{\C_{\xi}} &&&   U_q(\g). \\
}
\end{equation}
\end{enumerate}
It follows immediately from the definition of coalgebra anti-automorphism that 
\begin{equation}
\sigma^\xi := \Flip \circ (\xi_V^{-1} \otimes \xi_W^{-1}) \circ \xi_{V \otimes W}
\end{equation}
is an isomorphism of $U_q(\g)$ representations from $V \otimes W$ to $W \otimes V$. 

In the current work we require a little more freedom: we will sometimes use automorphisms $C_\xi$ of $U_q(\g)$ which are not linear over $\bc(q)$, but instead are bar-linear (i.e. invert $q$). This causes some technical difficulties, which we deal with in Section \ref{maketheta}. 

\begin{Comment}
To describe the data $( \C_\xi, \xi )$, it is sufficient to describe $C_\xi$, and the action of $\xi_{V_\lambda}$ on any one vector $v$ in each irreducible representation $V_\lambda$. This is usually more convenient then describing $\xi_{V_\lambda}$ explicitly. Of course, the choice of $C_\xi$ imposes a restriction on the possibilities for $\xi_{V_\lambda} (v)$, so when we give a description of $\xi$ in this way we are always claiming that the action on our chosen vector in each $V_\lambda$ is compatible with $C_\xi$.
\end{Comment}

\subsection{A useful lemma}
Let $(V_\lambda, v_\lambda)$ and $(V_\mu, v_\mu)$ be irreducible representations with chosen highest weight vectors. Every vector $u \in V_\lambda \otimes V_\mu$ can be written as
\begin{equation}
 u= v_\lambda \otimes c_0 + b_{k-1} \otimes c_1 + \ldots + b_1 \otimes c_{k-1} + b_0 \otimes v_\mu,
\end{equation}
where, for $0 \leq j \leq k-1$, $b_j$ is a weight vector of $V_\lambda$ of weight strictly less then $\lambda$, and $c_j$ a weight vector of $V_\mu$ of weight strictly less then $\mu$. Furthermore, the vectors $b_0 \in V_\lambda$ and $c_0 \in V_\mu$ are uniquely determined by $u$. Thus we can define projections from $V_\lambda \otimes V_\mu$ to $V_\lambda$ and $V_\mu$ as follows:

\begin{Definition} \label{p_def} The projections $p_{\lambda, \mu}^1: V_\lambda \otimes V_\mu \rightarrow V_\lambda$ and $p_{\lambda, \mu}^2: V_\lambda \otimes V_\mu \rightarrow V_\mu$ are given by, for all $u \in  V_\lambda \otimes V_\mu $, 
\begin{align} & p_{\lambda,\mu}^1(u):= b_0 \\
& p_{\lambda, \mu}^2(u):= c_0.
\end{align}
\end{Definition}

\begin{Lemma} \label{find_highest}
Let $S_{\lambda, \mu}$ be the space of singular vectors in $V_\lambda \otimes V_\mu$. 
The restrictions of the maps $p_{\lambda, \mu}^1$ and $p_{\lambda, \mu}^2$ from Definition \ref{p_def} to $S_{\lambda, \mu}$ are injective. \end{Lemma}

\begin{proof}
We prove the Lemma only for $p_{\lambda, \mu}^2$, since the proof for $p_{\lambda, \mu}^1$ is completely analogous. Let $c_1, \cdots c_m$ be a weight basis for $V_\mu$. Let $u$ be a singular vector of $V_\lambda \otimes V_\mu$ of weight $\nu$. Then $u$ can be written uniquely as
\begin{equation}
u = \sum_{j=1}^m v_j \otimes c_j,
\end{equation}
where each $v_j$ is a weight vector in $V_\lambda$. Let $\gamma$ be a maximal weight such that there is some $j$ with $\wt(v_j)=\gamma$ and $v_j \neq 0$. It suffices to show that $\gamma=\lambda$, so assume for a contradiction that it does not. Then $v_j$ is not a highest weight vector, so $E_i(v_j) \neq 0$ for some i. But then 
\begin{equation}
 E_i(u)=  \sum_{\wt(v_{j_s})=\gamma} E_i(v_{j_s}) \otimes c_{j_s} + \begin{aligned} & \text{ terms whose first factors have } \\ &\text{ weight strictly less then } \gamma + \alpha_i.
\end{aligned}
\end{equation}
Since the $c_j$ are linearly independent and $E_i(v_j) \neq 0$ for some $j$ with $\wt(v_j)=\gamma$, this implies that $E_i(u) \neq 0$, contradicting the fact that $v$ is a singular vector.\end{proof}

\section{Constructing the system of endomorphisms $\Theta$} \label{maketheta}
Constructing and studying $\Theta = \{\Theta_V\}$ is the technical heart of this work. As we mentioned in the introduction, $\Theta_V$ is bar linear instead of linear, which makes it more difficult to choose a normalization. To get around this, we introduce the notion of a bar involution $\barr_V$ on $V$, and actualy define $\Theta$ on the category of representations with a chosen bar involution. We then define a tensor product on this new category, and show that $(\Theta_{V, \barr_V}^{-1} \otimes \Theta_{W, \barr_W}^{-1}) \circ \Theta_{(V, \barr_V) \otimes (W, \barr_W)}$ does not depend on the choices of $\barr_V$ and $\barr_W$.
The real work is in defining this tensor product, which essentially amounts to defining a $\barr$ involution on $V \otimes W$ in terms of bar involutions $\barr_V$ and $\barr_W$.

\subsection{Bar involution} \label{makebar}
The following $\bq$ algebra involution of $U_q(\g)$ has been studied in several places, for example \cite[Section 1.3]{K}, and is usually called bar involution. We use the notation $C_\barr$ because we will also work with $\barr$ involutions $\barr_V$ on representations $V$, which are compatible with $C_\barr$ in the sense of Equation (\ref{compTheta}).

\begin{Definition} \label{bardef}
$C_\barr: U_q(\g) \rightarrow U_q(\g)$ is the $\bq$-algebra involution defined by
\begin{equation*}
\begin{cases}
C_\barr{q} = q^{-1}\\
C_\barr{K}_i = K_i^{-1}  \\
C_\barr{E}_i = E_i \\
C_\barr{F}_i = F_i.
\end{cases}
\end{equation*}
\end{Definition}

It is perhaps useful to imagine that $q$ is specialized to a complex number on the unit circle (although not a root of unity), so that $C_\barr$ is conjugate linear.

\begin{Definition} \label{barinvdef}
Let $V$ be a representation of $U_q(\g)$. A bar involution on $V$ is a $\bq$-linear involution $\barr_V$ such that
\begin{enumerate}
\item \label{barinv_comp} $\barr_V$ is compatible with $C_\barr$ in the sense that the following diagram commutes:
\begin{equation} 
\xymatrix{
V  \ar@(dl,dr) \ar@/ /[rrr]^{\barr_{V}} &&& V \ar@(dl,dr) \\
U_q(\g)  \ar@/ /[rrr]^{\C_{\barr}} &&&   U_q(\g). \\
}
\end{equation}
\item \label{bar_span} Let $V^{inv}= \{ v \in V \text{ such that } \barr_V(v)=v\}$. Then $V= \bq(q) \otimes_\bq V^{inv}$.

\end{enumerate}
 \end{Definition} 

\begin{Comment}
It is straightforward to check that $C_\barr^2$ is the identity. Along with condition (\ref{bar_span}), this implies that $\barr_V^2$ is the identity, so the term ``involution" is justified.  
\end{Comment}

\begin{Comment}
When it does not cause confusion we will denote $\barr_V (v)$ by $\bar{v}$. 
\end{Comment}

\begin{Proposition} \label{exists1} Fix $\lambda$ and a highest weight vector $v_\lambda \in V_\lambda$. There is a unique $\barr$ involution $\barr_{(V_\lambda, v_\lambda)}$ on $V_\lambda$ such that $\barr_{(V_\lambda, v_\lambda)}(v_\lambda)=v_\lambda$. 
\end{Proposition}

\begin{proof}
Recall that $V_\lambda$ has a basis consisting of various $F_{i_k} \cdots F_{i_1} v_\lambda$. All of these vectors must be fixed by any bar involution preserving $v_\lambda$, so there is at most one possibility. On the other hand, it is clear that the unique $\bq$-linear map sending 
$f(q) F_{i_k} \cdots F_{i_1} v_\lambda$ to $f(q^{-1}) F_{i_k} \cdots F_{i_1} v_\lambda$ for each of these basis vectors is a bar involution.
\end{proof}

\begin{Corollary} \label{exists}
Every representation $V$ has a (non-unique) bar involution $\barr_V$.
\end{Corollary}

\begin{proof}
Choose a decomposition of $V$ into irreducible components, and a highest weight vector in each irreducible component, then use Proposition \ref{exists1}.
\end{proof}

\begin{Definition} \label{bbd}
Fix $(V, \barr_V)$ and $(W, \barr_W)$, where $\barr_V$ and $\barr_W$ are involutions of $V$ and $W$ compatible with $C_\barr$. Let $(\barr_V \otimes \barr_W)$ be the vector space involution on $V \otimes W$ defined by $f(q) v \otimes w \rightarrow f(q^{-1}) \bar{v} \otimes \bar{w}$ for all $f(q) \in \bq(q)$ and $v \in V, w \in W$.
\end{Definition}

\begin{Comment}
It is straightforward to check that the action of $(\barr_V \otimes \barr_W)$ on a vector in $V \otimes W$ does not depend on its expression as a sum of elements of the form $f(q) v \otimes w$. The resulting map is a $\bq$-linear involution.
\end{Comment}

\begin{Definition} 
Fix $u \in V_\lambda \otimes V_\mu$ a weight vector of weight $\nu$. Define $v^\beta$ for each weight $\beta$ as the unique element of $V_\lambda(\nu-\beta) \otimes V_\mu(\beta)$ such that
\begin{equation}
u = \sum_{\text{weights } \beta} v^\beta.
\end{equation}
\end{Definition}

\begin{Lemma} \label{tensor_involution}
Fix $(V_\lambda, \barr_{V_\lambda})$ and $(V_\mu, \barr_{V_\mu})$. Let $v_\nu$ be a singular weight vector in $V_\lambda \otimes V_\mu$, and write
\begin{equation}
v_\nu= \sum_{j=1}^N b_j \otimes c_j,
\end{equation}
where each $b_j$ is a weight vector of $V_\lambda$, and each $c_j$ is a weight vector of $V_\mu$. Then
\begin{equation} \label{tbar}
\barr(v_\nu):= \sum_{j=0}^N q^{(\mu,\mu)-(\wt(c_j),\wt(c_j))+ 2(\mu-\wt(c_j), \rho)} \bar{b}_{j} \otimes  \bar{c}_{j}
\end{equation}
is also singular.
\end{Lemma}

\begin{proof}
Fix $i \in I$. The vector $v_\nu$ is singular, so $E_i v_\nu =0$ and hence $(E_i v_\nu)^\beta=0$ for all $\beta$. Then:
\begin{equation}
0= (E_i v_\nu)^\beta = \sum_{\wt(c_j) = \beta} q^{(\beta, \alpha_i)} E_i b_j \otimes c_j + \sum_{\wt(c_j)=\beta-\alpha_i} b_j \otimes E_i c_j.
\end{equation}
Using Equation (\ref{tbar}): 
\begin{align}
\nonumber  (E_i \barr(v_\nu))^\beta & = \sum_{\wt(c_j)  = \beta} q^{(\mu,\mu)-(\beta,\beta)+2(\mu-\beta, \rho)} q^{(\beta, \alpha_i)} E_i \bar{b}_j \otimes \bar{c}_j \\
& \hspace{0.2in} + \sum_{\wt(c_j)=\beta-\alpha_i} q^{(\mu,\mu)-(\beta-\alpha_i,\beta-\alpha_i)+2(\mu-\beta+\alpha_i, \rho)} \bar{b}_j \otimes E_i \bar{c}_j \\
&=  q^{(\mu,\mu)-(\beta-\alpha_i,\beta-\alpha_i)+2(\mu-\beta+\alpha_i, \rho)}  \times \\
\nonumber & \hspace{0.5in} \times \left(  \sum_{\wt(c_j) = \beta} q^{-(\beta, \alpha_i)} E_i \bar{b}_j \otimes \bar{c}_j + \sum_{\wt(c_j)=\beta-\alpha_i} \bar{b}_j \otimes E_i \bar{c}_j \right) \\
&=  q^{(\mu,\mu)-(\beta-\alpha_i,\beta-\alpha_i)+2(\mu-\beta+\alpha_i, \rho)}  (\barr_{V_\lambda} \otimes \barr_{V_\mu}) (E_i v_\nu)^\beta,
\end{align}
where $ (\barr_{V_\lambda} \otimes \barr_{V_\mu}) $ is the involution from Definition \ref{bbd}.
But $E_i(v_\nu)^\beta=0$, so we see that $E_i(v_\nu)^\beta=0$. Since this holds for all $i$ and all $\beta$, $\barr(v_\nu)$ is singular.
\end{proof}

\begin{Definition} \label{tb2}
Let $\barr_{(V_\lambda, v_\lambda) \otimes (V_\mu, v_\mu)}$ be the unique involution on $V_\lambda \otimes V_\mu$ which agrees with the involution $\barr$ from Lemma \ref{tensor_involution} on singular vectors, and is compatible with $C_\barr$.
\end{Definition}

\begin{Lemma} \label{isbar}
 $\barr_{(V_\lambda, v_\lambda) \otimes (V_\mu, v_\mu)}$ is a bar involution.
\end{Lemma}

\begin{proof}
Definition \ref{barinvdef} part (\ref{barinv_comp}) follows immediately from the definition of $\barr_{(V_\lambda, v_\lambda) \otimes (V_\mu, v_\mu)}$. To establish Definition \ref{barinvdef} part (\ref{bar_span}), it suffices to show that there is a basis for the space $S_{\lambda, \mu}$ of singular vetors of $V_\lambda \otimes V_\mu$ which is fixed by  $\barr_{(V_\lambda, v_\lambda) \otimes (V_\mu, v_\mu)}$. Since $V_\lambda= \bq(q) \otimes_\bq V_\lambda^{inv}$, there is a basis for $S_{\lambda, \mu}$ consisting of elements of $V_\lambda^{inv} \otimes_\bq V_\mu$. Using Lemma \ref{find_highest}, we see that there is a basis for $S_{\lambda, \mu}$ consisting of vectors of the form
\begin{equation}
v_\lambda \otimes c_0 + \cdots + b_0 \otimes v_\mu,
\end{equation}
where $\bar{b}_0 = b_0$ and the missing terms are all of the form $b \otimes c$ with $\wt(c)<\mu$. By Definition \ref{tb2} and Lemma \ref{find_highest}, this vector is invariant under $\barr_{(V_\lambda, v_\lambda) \otimes (V_\mu, v_\mu)}$. \end{proof}

In light of Definition \ref{barinvdef} part (\ref{bar_span}), we can extend Definition \ref{tb2} by naturality to construct a $\barr$-involution on $(V, \barr_V) \otimes (W, \barr_W)$ in terms of any $\barr$-involutions $\barr_V$ and $\barr_W$. 

\subsection{The system of endomorphisms $\Theta$}

Consider the $\bq$-algebra automorphism $C_\Theta$ of $U_q(\g)$: 
\begin{equation}
\begin{cases}
C_\Theta (E_i) =   E_{i} K_i^{-1} \\
C_\Theta (F_i) =   K_i F_{i} \\
C_\Theta (K_i) = K_{i}^{-1} \\
C_\Theta (q) = q^{-1}.
\end{cases}
\end{equation}
Notice that $C_\Theta$ is not linear over $\bq(q)$, but instead inverts $q$. One can easily check that $C_\Theta$ is a $\bq$ algebra involution, and that it is also a coalgebra anti-involution. 

 \begin{Definition} \label{Thetadef}
 Fix a representation $V$ with a bar involution $\barr_V$. Then $\Theta_{V, \barr_V}$ is the $\bq$ linear endomorphism of $V$ defined by 
 \begin{equation}
 \Theta_{V, \barr_V}(v)= q^{-( \wt(v), \wt(v))/2 + ( \wt(v), \rho )} \barr_V(v).
 \end{equation}
\end{Definition}

\begin{Comment} \label{thetarestricts}
Using Definitions \ref{bardef}, one can see that, for any irreducible $V_\lambda \subset V$, $\Theta_{V, \barr_V}$ restricts to an endomorphism of $V_\lambda$. 
\end{Comment}

\begin{Comment} There are sometimes weights $\lambda$ for which
$-( \lambda, \lambda)/2 + ( \lambda, \rho )$ is not an integer. However, it is always a multiple of $1/k$ where $k$ is twice the size of the weight lattice mod the root lattice. It is for this reason that we adjoin $q^{1/k}$ to the base field. 
\end{Comment} 

\begin{Lemma} \label{thetaisok}
the following diagram commutes
\begin{equation}
\xymatrix{
V \ar@(dl,dr) \ar@/ /[rrr]^{\Theta_V} &&& V \ar@(dl,dr) \\
U_q(\g)  \ar@/ /[rrr]^{\C_{\Theta}} &&&   U_q(\g). \\
}
\end{equation}
\end{Lemma}

\begin{proof} 
It is sufficiant to check that $C_\Theta(X) \Theta_V(v) = \Theta_V(X v)$, where $X=E_i$ or $F_i$. We do the case of $F_i$ and leave $E_i$ as an excersize. Fix $v \in V$.
\begin{align}
\Theta_V(F_i v) &= q^{-( \wt(F_iv), \wt(F_iv))/2 + ( \wt(F_iv), \rho )} \barr_V(F_iv) \\
&= q^{-( \wt(v)-\alpha_i, \wt(v)-\alpha_i)/2 + ( \wt(v)-\alpha_i, \rho )} F_i \barr_V(v) \\
&\label{middleeq} = q^{(\alpha_i, \wt(v)-\alpha_i)}  q^{-( \wt(v), \wt(v))/2 + ( \wt(v), \rho )} F_i \barr_V(v) \\
&=   K_i  F_i q^{-( \wt(v), \wt(v))/2 + ( \wt(v), \rho )}\barr_V(v) \\
&= C_\Theta(F_i) \Theta_V(v). 
\end{align}
where for Equation (\ref{middleeq}) we have used the fact that $(\alpha_i, \alpha_i)/2= (\alpha_i, \rho)=d_i$.   
\end{proof}

\begin{Definition}
Fix two representations with bar involutions $(V, \barr_V)$ and $(W, \barr_W)$. We set $\Theta_{(V, \barr_V) \otimes (W, \barr_W)}$ to be the $\bq$ linear endomorphism of $V \otimes W$ defined by, for all $u \in V \otimes W$,
\begin{equation}
\Theta_{(V, \barr_V) \otimes (W, \barr_W)} (u) = q^{-( \wt(u), \wt(u))/2 + ( \wt(u), \rho )} \barr_{(V \barr_V) \otimes (W, \barr_W)}.
\end{equation}
\end{Definition}

\begin{Comment}
By Lemma \ref{isbar}, $\barr_{(V \barr_V) \otimes (W, \barr_W)}$ is a bar involution on $V \otimes W$, so by Lemma \ref{thetaisok}, $\Theta_{(V, \barr_V) \otimes (W, \barr_W)} $ is compatible with $C_\Theta$.
\end{Comment}

\section{Main Theorem} \label{main}

\begin{Theorem} \label{main_th}
$(\Theta_{V, \barr_V}^{-1}  \otimes \Theta_{W, \barr_W} ^{-1}) \Theta_{(V \otimes \barr_V) \otimes (W, \barr_W)}$ acts on $V \otimes W$ as the standard $R$-matrix. This holds independent of the choice of bar involutions $\barr_V$ and $\barr_W$.
\end{Theorem}

\begin{proof} We will actually prove the equivalent statement that
\begin{equation}
\sigma^\Theta:= \Flip \circ (\Theta_{V, \barr_V}^{-1}  \otimes \Theta_{W, \barr_W} ^{-1}) \Theta_{(V \otimes \barr_V) \otimes (W, \barr_W)}
\end{equation}
acts on $V \otimes W$ as the standard braiding $\Flip \circ R$. By  Lemma \ref{thetaisok} and the fact that $C_\Theta$ is a $\bq$ coalgebra anti-automorphism, the following diagram commutes:
\begin{equation}
\xymatrix{
V \otimes W \ar@(dl,dr) \ar@/ /[rrrr]^{ \Theta_{(V \otimes \barr_V) \otimes (W, \barr_W)}} &&&& V \otimes W \ar@(dl,dr)  \ar@/ /[rrrrr]^{ \Flip \circ (\Theta_{V \otimes \barr_V}^{-1} \otimes \Theta_{W, \barr_W}^{-1})} &&&&& W \otimes V \ar@(dl,dr)\\
U_q(\g)  \ar@/ /[rrrr]^{\C_{\Theta}} &&&&   U_q(\g)   \ar@/ /[rrrrr]^{\C_{\Theta}^{-1}} &&&&&   U_q(\g). \\
}
\end{equation}
In particular, $\sigma^\Theta: V \otimes W \rightarrow W \otimes V$ is an isomorphism. Thus it suffices to show that  $\sigma^\Theta(v_\nu)= \Flip \circ R(v_\nu)$ for every singular weight vector $v_\nu \in V \otimes W$. By naturality it is enough to consider the case when $V$ and $W$ are irreducible, so let $v_\nu$ be a singular vector in $V_\lambda \otimes V_\mu$. Write 
\begin{equation}
v_\nu= b_\lambda \otimes c_0 + b_{k-1} \otimes c_1 + \ldots + b_1 \otimes c_{k-1} + b_0 \otimes b_\mu,
\end{equation}
where for $0 \leq j \leq k-1$, $b_j$ is a weight vector of $V_\mu$ of weight strictly less then $\mu$. By Definitions \ref{tb2} and \ref{Thetadef}, 
\begin{align}  \sigma^\Theta(v_\nu) 
&= \Flip \circ (\Theta_{V_\lambda, \barr_{V_\lambda}}^{-1} \otimes \Theta_{V_\mu, \barr_{V_\mu}}^{-1}) \Theta_{(V_\lambda, \barr_{V_\lambda}) \otimes (V_\mu, \barr_{V_\mu})}(\cdots +  b_0 \otimes b_\mu ) \\
&= \Flip \circ (\Theta^{-1}_{V_\lambda, \barr_{V_\lambda}} \otimes \Theta^{-1}_{V_\mu, \barr_{V_\mu}})
(q^{-(\mu+\wt(b_0), \mu+\wt(b_0))/2+(\mu+\wt(b_0),\rho) }( \ldots + \bar{b}_0 \otimes \bar{b}_\mu)) \\
 &= q^{-(\wt(b_0), \wt(b_0))/2  -(\mu, \mu)/2+(\mu+\wt(b_0), \mu+\wt(b_0))/2 } b_\mu \otimes b_0 + \ldots \\
& =q^{(\wt(b_0), \mu)}  b_\mu \otimes b_0 + \ldots,
\end{align}
where $\ldots$ always represents terms where the factor coming from $V_\mu$ has weight strictly less then $\mu$.
It follows immediately from Proposition \ref{goodR} that 
\begin{equation}
\Flip \circ R (v_\nu)= q^{(\wt(b_0), \mu)} b_\mu \otimes b_0 + \ldots,
\end{equation}
where again $\ldots$ represents terms of the form $c \otimes b$ where $\wt(c) < \mu$. Both $\sigma^\Theta(v_\nu)$ and $\Flip \circ R (v_\nu)$ are singular vectors in $V_\mu \otimes V_\lambda$, so by Lemma \ref{find_highest} they are equal.
\end{proof}

\begin{Comment}
The above proof works independent of the choice of $\barr_V$ and $\barr_W$. One can also see directly that $\sigma^\Theta$ does not depend on these choices. Restrict to the irreducible case, and notice that by Lemma \ref{exists1}, $\sigma^\Theta$ depends only the a choice of highest weight vectors $v_\lambda$ and $v_\mu$. It is straightforward to check that rescaling these vectors has no effect on $\sigma^\Theta$.
\end{Comment}

\begin{Comment}
One can check that $\Theta_V$ is an involution of $\bq$ vector spaces, so the inverses in the statement of Theorem \ref{main} are in some sense unnecessary. We include them because $\Theta_{V}$ should really be thought of as an isomorphism between $V$ and the module which is $V$ as a $\bq$ vector space, but with the action of $U_q(\g)$ twisted by $C_\Theta$. We have not specified the action of $\Theta$ on this new module. The way the formula is written, $\Theta$ is always acting on $V, W$ or $V \otimes W$ with the usual action, where it has been defined.  \end{Comment}

\section{Future directions} \label{Questions}

We have two main motivations for developing our formula for the R-matrix. 

\begin{Motivation}
In work with Joel Kamnitzer \cite{Rcommutor}, we showed that Drinfeld's unitarized R-matrix $\bar{R}$ (see \cite{D}) respects crystal basis (up to some signs). Composing with $\Flip$, we see that $\bar{R}$ descends to a crystal map from $B \otimes C$ to $C \otimes B$, which is fact agrees with the crystal commutor defined in \cite{cactus}. We make extensive use of Equation (\ref{KReq}), so our methods are only valid in the finite type case. However Drinfeld's unitarized R-matrix is defined in the symmetrizable Kac-Moody case, as is the crystal commutor (see \cite{KT1} and \cite{Savage}). We hope that the formula given in Theorem \ref{main_th} will help us to extend some of the results in \cite{Rcommutor} to the symmetrizable Kac-Moody case. 
\end{Motivation}

\begin{Motivation}
Recall that the action of the braiding $\Flip \circ R$ on $V \otimes W$ can be drawn diagrammatically as passing a string labeled $V$ over a string labeled $W$. If we use flat ribbons in place of strings, as it is often convenient to do, one can consider the following isotopy:

\vspace{0.05cm}
\setlength{\unitlength}{0.55cm}
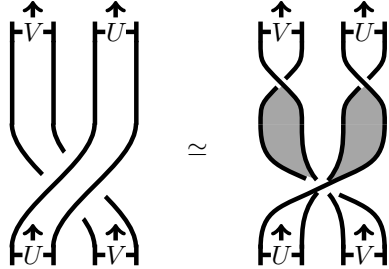
\begin{center}
\begin{picture}(10,6)
\put(6,0.2){\Dtst}
\put(6,3.2){\ptwist}
\put(8,3.2){\ptwist}
\put(6,-0.7){\puuv[U]}
\put(8,-0.7){\puuv[V]}
\put(6,4.7){\puuv[V]}
\put(8,4.7){\puuv[U]}

\put(4.5,2.5){$\simeq$}
\put(-0.18,-0.44){\pplaincrossing}
\put(0,3.2){\uid}
\put(2,3.2){\uid}
\put(0,-0.7){\puuv[U]}
\put(2,-0.7){\puuv[V]}
\put(0,4.7){\puuv[V]}
\put(2,4.7){\puuv[U]}

\end{picture}
\end{center}
\vspace{0.1cm}
Roughly, if one interprets twisting a ribbon by 180 degrees as $X$, and twisting two ribbon together as at the bottom on the right side as $\Flip \circ \Delta(X)$, the two sides of this isotopy correspond to the two sides of Equation (\ref{KReq}), written as
\begin{equation}
\Flip \circ R= \Flip \circ (X^{-1} \otimes X^{-1}) \Delta(X) =(X^{-1} \otimes X^{-1}) \circ \Flip \circ \Delta(X).
\end{equation}
In work with Noah Snyder \cite{half_twist}, we make this precise. 
One should be able to use our new formula to give a precise interpretation of ``twisting a ribbon by 180 degrees" in the symmetrizable Kac-Moody case. It is for this reason that we use the term ``half twist type formula" in our title. 
\end{Motivation}

\end{document}